\documentclass[11pt, reqno]{amsart}
 
\usepackage[cp1251]{inputenc}
\usepackage{amsmath,amsthm,mathrsfs}
\usepackage{amssymb}
\usepackage[english]{babel}
\usepackage{esint}
\usepackage[titletoc]{appendix}

\usepackage{graphicx}
\usepackage{float}
\usepackage{subfigure}

\usepackage{bbm}
\usepackage{amsfonts,amstext,amssymb,verbatim,epsfig}
\usepackage{pgf,tikz}
\usepackage{float}
\usetikzlibrary{arrows}
\usepackage{hyperref}
\usepackage{cite}
\usepackage{epstopdf}
\usepackage{color}
\usepackage{transparent}
\usepackage{subfigure}
\usepackage{pgfplots}
\usepgfplotslibrary{polar}
\usepgflibrary{shapes.geometric}
\usetikzlibrary{calc}

\def\Z{{\mathbb{Z}}}

\usepackage{tikz}

\usepackage{xcolor}
\usepackage[normalem]{ulem}

\hypersetup{colorlinks=true,pdfborder=001,  citecolor  = blue, citebordercolor= {magenta}}
\hypersetup{linkcolor=magenta, linkbordercolor = blue}
\hypersetup{colorlinks,urlcolor=blue}
\makeatletter
\let\reftagform@=\tagform@
\def\tagform@#1{\maketag@@@{(\ignorespaces\textcolor{magenta}{#1}\unskip\@@italiccorr)}}
\renewcommand{\eqref}[1]{\textup{\reftagform@{\ref{#1}}}}
\makeatother

\makeatletter

\DeclareUrlCommand\ULurl@@{%
  \def\UrlLeft{\uline\bgroup}%
  \def\UrlRight{\egroup}}
\def\ULurl@#1{\hyper@linkurl{\ULurl@@{#1}}{#1}}
\DeclareRobustCommand*\ULurl{\hyper@normalise\ULurl@}
\makeatother

\def\lessim{\ \lower4pt\hbox{$
		\buildrel{\displaystyle <}\over\sim$}\ }
\def\gessim{\ \lower4pt\hbox{$\buildrel{\displaystyle >}
		\over\sim$}\ }

\newtheorem{theorem}{\bf Theorem}
\newtheorem{lemma}[theorem]{\bf Lemma}

\newtheorem{example}{\bf Example}

\newtheorem{proposition}[theorem]{\bf Proposition}

\theoremstyle{remark}
\newtheorem{remark}{Remark}

\newenvironment{Proof of lemma}{\noindent{\bf Proof of Lemma}}{\hfill$\Box$\newline}
\newenvironment{Proof of theorem}{\noindent{\bf Proof of Theorem}}{\hfill{\footnotesize${\square}$}\newline}
\newenvironment{Proof of theorems}{\noindent{\bf Proof of Theorems}}{\hfill$\Box$\newline}
\newenvironment{Proof of proposition}{\noindent{\bf Proof of Proposition}}{\hfill$\Box$\newline}
\newenvironment{Proof of propositions}{\noindent{\bf Proof of Propositions}}{\hfill$\Box$\newline}
\newenvironment{Proof of exercise}{\noindent{\it Proof of Exercise:}}{\hfill$\Box$}
\begin{document}
\title{On the Gardner Transition in the Ising pure $p$-Spin Glass}


\author{Yuxin Zhou}
\address{Department of Statistics, University of Chicago}
\email{yuxinzhou@uchicago.edu}

\begin{abstract}
We prove the existence of one-step replica symmetry breaking (1RSB) for the mean field
Ising spin glasses at finite temperature and identify the first critical temperature in Gardner transition. Specifically,
Gardner conjectured that for Ising pure $p$-spin glasses with  $p\geq 3,$ there are two phase transitions: the Parisi measure is Replica Symmetric (RS) at high temperatures, then transitions to One-Step Replica Symmetry Breaking (1RSB) at a lower temperature, and finally reaches Full Replica Symmetry Breaking (FRSB) at very low temperatures. Our main results verify the first part of the Gardner transition: for any $p \geq 3,$ there exists a unique $\beta^p_1>0$ such that the Parisi measure is RS for $0 < \beta \leq \beta^p_{1}$ and then 1RSB for $\beta^p_1 < \beta \leq \beta^p_1+\epsilon_p$ with some $\epsilon_p>0.$ We also provide a computational method to locate $\beta^p_1$ for any $p \geq 3.$
\end{abstract}

\maketitle
\section{Introduction and main results}

The Ising pure $p$-spin glass is a crucial example of mean field spin glasses, resulting in a broad array of problems and phenomena in both physical and mathematical sciences. For detailed information on its background, history, and methods, the books by Mezard, Parisi, and Virasoro \cite{mezard1987spin}, as well as Talagrand \cite{talagrand2006parisi-b} and their extensive references, are recommended.

In this paper, we investigate the structure of the functional order parameter for the mean field Ising pure $p$-spin model in the absence of external field. This order parameter, referred to as the Parisi measure, is expected to provide a comprehensive qualitative description of the system and has been extensively studied by researchers in both physics and mathematics \cite{mezard1987spin, talagrand2006parisi-b}.  Recent discoveries have shed light on Parisi measures in \cite{AChen15PTRF,Aukosh17PTRF}, yet the structure of these measures remains elusive. 
In previous studies of mean field Ising spin glasses, only  Parisi measures that are replica symmetric have been rigorously proven to exist in specific models  and no rigorous examples  beyond the replica symmetric phase have been established.


We now introduce the mean field Ising pure $p$-spin model. 
Let $p,N$ be integers with $p \geq 2$ and $N\geq1$. For any $N \geq 1$,  let $\Sigma_N:= \{ -1,+1 \}^N$ be the Ising spin configuration space. The Hamiltonian of the mean field Ising pure $p$-spin model  is a Gaussian function defined as
\begin{eqnarray*}
H_N(\sigma):=\frac{\beta}{N^{\frac{p-1}{2}}} \sum_{1 \leq i_1,\cdots,i_p \leq N} g_{i_1,\dots,i_p} \sigma_{i_1} \cdots \sigma_{i_p},
\end{eqnarray*}
for $\sigma=(\sigma_1,\cdots,\sigma_N) \in \Sigma_N,$
where all $(g_{i_1,\cdots, i_p})$, $1 \leq i_1,\cdots i_{p} \leq N$,  are independent, identically distributed standard Gaussian random variables.  Here $\beta>0$ is a parameter of the Ising pure $p$-spin model, called the inverse temperature.  When $p=2,$ the model is also called the Sherrington-Kirkpatrick (SK) model. In this paper, we mainly focus on the models with $p \geq 3.$

The Gaussian field $H_{N}$ is centered with covariance given by 
\begin{eqnarray*}
\mathbb{E} H_N(\sigma^1) H_N(\sigma^2)=N \xi(R_{1,2})
\end{eqnarray*}
where  $R_{1,2}:=\frac{1}{N}\sum^N_{i=1} \sigma_i^1 \sigma^2_i$
is the normalized inner product between $\sigma^1$ and $\sigma^2$ and 
\begin{eqnarray}\label{eq:psxi}
\xi(x):=\beta^2  x^p.
\end{eqnarray}

 Thanks to the groundbreaking work of Parisi \cite{parisi1979infinite,parisi1980infinite}, it was predicted that the thermodynamic limit of the free energy
can be computed by a variational formula and this prediction was subsequently rigorously confirmed by  Panchenko and Talagrand\cite{panchenko, talagrand2006parisi}. 
  To be more specific, denote the space $M[0,1]$ of all probability measures on $[0,1]$ and then consider the Parisi functional
$\mathcal{P}$ (see \eqref{Parisifunctional}) defined on $M[0,1].$ The following limit exists almost surely,
\begin{eqnarray*}
\lim_{N \rightarrow \infty} \frac1N \log {\sum_{\sigma \in \Sigma_N}  \exp H_N(\sigma)}=\inf_{\mu \in M[0,1]} \mathcal{P} (\mu).
\end{eqnarray*}

As an infinite dimensional variational formula, $\mathcal{P}$ is continuous and always has a minimizer. The uniqueness of the minimizer is first proven by  Auffinger and Chen\cite{{auffinger2015parisi}}. This unique minimizer of $\mathcal{P}$ is called the Parisi measure, denoted by $\mu_P$. The Parisi measure 
$\mu_P$ is essential for characterizing the mean field spin glass models. In contrast, the structure of the Parisi measures remains mysterious. 

We first present the classification of the structure of the Parisi measure $\mu_P$.  We say that the Parisi measure $\mu_P$ is Replica Symmetric (RS) if it's a Dirac measure; One Replica Symmetric Breaking (1RSB) if it consists of two atoms; Full Replica Symmetric
Breaking (FRSB) if its support contains some interval. 

To the best of our knowledge, only the structures of Parisi measures at high temperature in the SK model, i.e. $p=2$  and $0 < \beta < \frac{1}{\sqrt{2}}$  have been classified rigorously   by Aizenman, Lebowitz and
Ruelle in \cite{ALR}, which  are RS. As for low temperature, Toninelli \cite{Toni} showed  that the Parisi measure is not
RS for $\beta>\frac{1}{\sqrt{2}}$.   Beyond that,  there were no rigorous examples of Ising spin glass models other than the replica symmetric phase. Moreover, for $p \geq 3$,  no phase transitions between replica symmetric and non-replica symmetric phases were  proved rigorously. See also \cite{Aukosh17PTRF} for results in the $p=2$ setting with an external field.

Only a few predictions about Parisi measures have been presented in the physics literature. The following discussions represent the most widely acknowledged predictions concerning the structure of Parisi measures. In the model setting of this paper, for $p \geq 3,$ it's conjectured by Gardner \cite{Gardner} that there exists two phase transitions, denoted by $\beta^p_{1}$ and $\beta^p_{2}$. Firstly, the Parisi measure is RS at high temperature, i.e. $\mu_P=\delta_0$, for $0<\beta \leq \beta^p_{1}$. Then   the Parisi measure is 1RSB at relatively low temperature, i.e. $\mu_P=m \delta_0 +(1-m) \delta_{q},$ where $0<m,q<1$, for $\beta^p_{1} < \beta \leq \beta^p_{2}, $ Lastly,  the Parisi measure is FRSB at extremely low temperature, i.e. $\mu_P=m \delta_0 +\nu +(1-m') \delta_q$, where $\nu$ is fully supported in an interval and has a smooth density, for $\beta>\beta^p_{2}$.

 Our main results below verify the existence of the 1RSB phase and the first phase transition $\beta^p_{1}$, for $p \geq 3$:
 \begin{theorem}\label{mainthm}
 For $p \geq 3,$ there exists $\beta^p_1,\epsilon_p>0$ such that the following holds:
 \begin{enumerate}
 \item  the Parisi measure is RS, i.e. $\mu_P=\delta_0$ if and only if if $\beta$ is in $(0, \beta^p_1],$
 \item For $\beta \in (\beta^p_1,  \beta^p_1+\epsilon_p], $ the Parisi measure is 1RSB, i.e. $\mu_P=m \delta_0 +(1-m) \delta_{q}$ for some $0<m,q<1.$
 \end{enumerate}
 
 
 \end{theorem}
 
  \begin{figure}[H] 
\centering 
\includegraphics[width=0.9\textwidth]{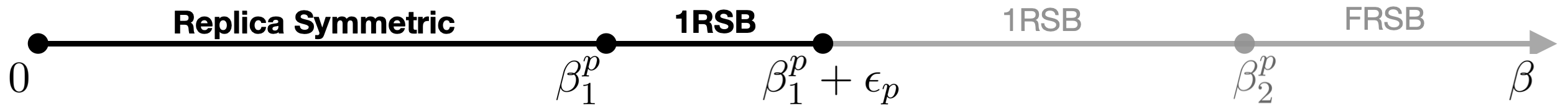} 
\caption{Phase transitions of  $\mu_P$ with respect to  $\beta$. The phases in black are the main results in Theorem \ref{mainthm} and the phases in grey remains unknown.} 
\label{Fig4}
\end{figure}

The relation between the phases of the Parisi measure $\mu_P$ and the inverse temperature $\beta$ is illustrated in Figure \ref{Fig4}. The first phase in black represents the exact phase that the Parisi measure is replica symmetric and the corresponding temperature interval is $(0,\beta^p_1]$.  The second phase in black represents the phase that the Parisi measure is 1RSB, where the corresponding interval is $(\beta^p_1, \beta^p_1+\epsilon_p].$ The first phase in grey is  conjectured to be 1RSB for $\beta \in (\beta^p_1+\epsilon_p,\beta^p_2]$ and the second to be FRSB for $\beta >\beta^p_2.$ Both phases in grey still remains mysterious to us.

The rest of the paper is organized as follows. In Section \ref{section2}, we explain our main results in Theorem \ref{mainthm} in more details. To be more specific, in Section \ref{1stphase}, we present  a more detailed explanation for the phase transition of Parisi measures between RS and 1RSB. In Section \ref{proofidea}, we present the precise characterization of the first critical temperature $\beta^p_1$ in \eqref{boundary}.  Moreover, we introduce the derivation of this characterization and sketch the proof ideas briefly. In section \ref{computational}, we  provide a computational method to find $\beta^p_1$ in Proposition \ref{criterion} and  offer several examples of phase transitions from RS to 1RSB.  In Section \ref{section4}, we present the definition of the Parisi functional and present a general characterization for the Parisi measure. We then prove our main results in Section \ref{section3}. Specifically, we  establish the first part of Theorem \ref{mainthm} and Proposition \ref{criterion} in Section \ref{section3.1}. A crucial ingredient in the proof of the main results is Lemma  \ref{tu}. Lastly, we prove the second part in Section \ref{section3.2}.  

\section{The  phase transition of Parisi measures between RS and 1RSB}\label{section2}

In this section, we explain our main results about the first critical temperature and phase transition in Theorem \ref{mainthm} in more details. 
\subsection{The phase transition between RS and 1RSB}\label{1stphase}

 We denote the distribution of the Parisi measure $\mu_P$ by $\alpha_{\mu_P}$, i.e. $\mu_P([0,s])=\alpha_{\mu_P}(s)$ for $s \in [0,1].$
  Based on our main results in Theorem \ref{mainthm}, distributions of $\mu_P$ at different inverse temperatures $\beta \in (0,\beta^p_1+\epsilon_p]$ with $p \geq 3$ are illustrated in Figure \ref{Fig3}.  
  
   \begin{figure}[H] 
\centering 
\includegraphics[width=0.95\textwidth]{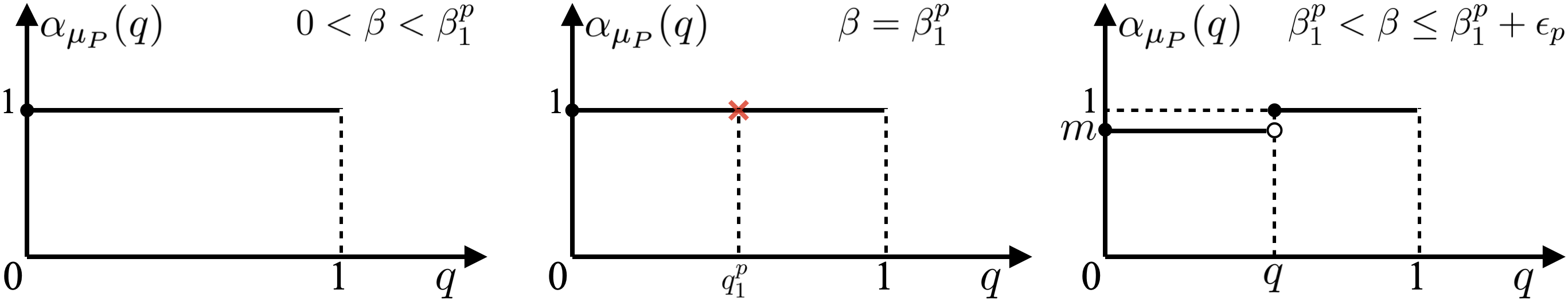} 
\caption{Distributions $\alpha_{\mu_P}$ of  Parisi measures $\mu_P$ for the Ising pure $p$-spin model with $p \geq 3$ and $\beta \in (0,\beta_p+\epsilon_p]$. The  figures from left to right are distributions of $\alpha_{\mu_P}$ for $\beta \in (0,\beta^p_1),$ $\beta=\beta^p_1$ and $\beta \in (\beta^p_1,\beta^p_1+\epsilon_p]$, respectively.} 
\label{Fig3}
\end{figure}
  
  At high temperature region, the Parisi measure is RS, i.e. when $\beta \in (0,\beta^p_1]$, $\mu_P=\delta_0$ and then $\alpha_{\mu_P}(s)=\mathbbm{1}_{[0,1]}(s).$  At the first critical temperature $\beta=\beta^p_1,$ the Parisi measure is on the verge of transitioning from RS to 1RSB phase. Indeed when $\beta=\beta^p_1$, $\mu_P$ can be regarded as the threshold of RS  and 1RSB as follows: 
 \begin{eqnarray*}
 \mu_P&=&\delta_0 \text{ (RS), }\\
 &=&m\delta_0+(1-m)\delta_{q^p_1}, \text{ where } m=1 \text{ and }q^p_1\text{ in \eqref{boundary}} \text{ (1RSB). }
 \end{eqnarray*}
 Here the first equality is the usual way to regard $\mu_P$ as replica symmetric while the second equality is the way to regard it as a special case of 1RSB since $\mu_P$ is about to have two atoms.  Therefore, the transition of Parisi measures is  continuous with respect to the inverse temperature $\beta$. 
 
  The distribution can also be regarded as the threshold of RS and 1RSB: 
 \begin{eqnarray*}
 \alpha_{\mu_P}(s)&=&\mathbbm{1}_{[0,1]}(s) \text{ (RS), }\\
 &=&m \mathbbm{1}_{[0,q^p_1)}(s)+  \mathbbm{1}_{[q^p_1,1]}(s), \text{ where } m=1 \text{ (1RSB). }
 \end{eqnarray*}
  As for the figure of $\alpha_{\mu_P}$ at $\beta=\beta^p_1,$ the only step in $[0,1]$ is about to break into two steps $[0,q^p_1)$ and $[q^p_1,1]$.  
  
  At relatively low temperature, the Parisi measure transitions to 1RSB phase, i.e. when $\beta \in (\beta^p_1,\beta^p_1+\epsilon_p]$, $\mu_P=m\delta_0+(1-m)\delta_{q},$ for some $m,q \in (0,1).$ Then the distribution of $\mu_P$ is $$\alpha_{\mu_P}(s)=m \mathbbm{1}_{[0,q)}(s)+  \mathbbm{1}_{[q,1]}(s).$$

\subsection{Characterization of $\beta^p_1$ and proof ideas}\label{proofidea}

In this section, we provide the precise characterization of $\beta^p_1$ and sketch the proof ideas briefly. 

Consider the following equations with respect to $(\beta,q)$,
\begin{eqnarray}\label{boundary}
\left\{
\begin{array}{lcl}
C_\beta(q):=\frac{ \mathbb{E} [ \cosh (\sqrt{\xi'(q)} g) \log \cosh (\sqrt{\xi'(q)} g) ] }{\mathbb{E} [ \cosh (\sqrt{\xi'(q)} g) ] }-\frac12 \xi'(q) -\frac12 \theta(q)=0 \\
D_\beta(q):=\frac{ \mathbb{E} [ \tanh^2 (\sqrt{\xi'(q)} g)  \cosh (\sqrt{\xi'(q)} g) ] }{\mathbb{E} [ \cosh (\sqrt{\xi'(q)} g) ] }- q=0
\end{array} \right.,
\end{eqnarray}
where $g$ is a standard Gaussian random variable and $\theta(q):=q \xi'(q)-\xi(q)$.
We claim that for $p \geq 3,$ the solution to $\eqref{boundary}$ exists and is unique. We defer the proof to Section \ref{section3.1} and denote this unique solution by $(\beta^p_1,q^p_1)$. Then  the  phase transition between RS and 1RSB occurs when $\beta=\beta^p_1$. 

Now we sketch our proof ideas briefly where the system of equation \eqref{boundary} will  come up naturally.

In Section \ref{section4}, we will introduce a necessary and sufficient criterion to characterize the Parisi measure. To be more specific, for each $\mu \in \text{M}[0,1],$ we define a function $f_\mu(u)$ on $[0,1]$ (see \eqref{criterionfunc2}. 
Then the Parisi measure $\mu_P$ is uniquely characterized by $f_{\mu_P}(u) \leq 0, \forall u \in [0,1]$ and $f_{\mu_P}(u) = 0, \forall u \in \text{supp}(\mu_P)$. 

 In Section \ref{section3.1}, we apply this criterion to the RS case.  It follows from the definition that $f_{\delta_0}(u)=C_\beta(u), $ for $u \in [0,1]$.
 Therefore, $\mu_P=\delta_0$ is the Parisi measure if and only if  $C_{\beta}(u) \leq 0, \forall u \in [0,1].$ Moreover, the phase transition between RS and 1RSB phases will occur when $C_\beta(u)=0$ for some $u \in (0,1).$  Also, we will see that when $u \geq 0$, $D_\beta(u)$ has the same sign as the derivative of $C_\beta(u)$.  Therefore, the phase transitions from RS to 1RSB when $D_\beta(u)=0$ and $C_\beta(u)=0$, i.e. \eqref{boundary} holds. 
 
 We will see that this occurs precisely when $\beta = \beta_1^p$ and $u = q_1^p$.
Indeed when $\beta$ gets larger than $\beta_1^p$, we will see $C_\beta$ is no longer non-positive on $(0, 1)$ and hence the Parisi measure is no longer RS. 
 
 
 In Section \ref{section3.2}, we apply this criterion to the 1RSB case and verify that there exists $\epsilon_p>0,$ such that the Parisi measure $\mu_P$ is 1RSB for $\beta \in (\beta^p_1,\beta^p_1+\epsilon_p].$
 We will see that when $\beta$ is near $\beta_1^p$, the formulas of the RS criterion and the 1RSB criterion are close to each other. This will enable us to use the method of continuity to establish our main results.

\begin{remark}[Comparison to the SK model] When $p=2,$ for any $\beta>0$, there is no $q \in (0,1)$ such that $D_\beta(q)=0$ and $C_\beta(q)=0$ simultaneously.  Therefore the system of equations \eqref{boundary} has no solution in $(0,+\infty) \times (0,1)$ and thus there is no phase transition between RS and 1RSB phases in the SK model. 

The above claim can be seen by arguments similar to those in Section \ref{section3.1}: We first notice $\xi'''(u)\equiv0$. Then by a similar computation to \eqref{similarcomp}, we obtain that $\frac{d^2}{du^2} \{ D_\beta(u) \} <0, \forall u \in [0,+\infty).$ Thus $D_{\beta}(u)$ cannot have two critical points on $(0, 1)$.

\end{remark}

\subsection{A computational criterion to locate $\beta^p_1$}\label{computational}
We now provide the following criterion to determine the range of $\beta^p_1$ for $p \geq 3.$ 
\begin{proposition}\label{criterion}
For $p \geq 3$ and $\beta_1,\beta_2>0$, $\beta_1^p$ lies in  $[\beta_1,\beta_2]$  if the following holds:
\begin{enumerate}
\item there exists $0 <q^1_0<q^2_0<q^1_1<q^2_1<1$ such that 
\begin{enumerate}
\item $D_{\beta_1}(q^1_0) > D_{\beta_1}(q^2_0)>0>D_{\beta_1}(q^1_1) > D_{\beta_1}(q^2_1)$,
\item $C^1_{\beta_1}(q^1_1) - C^2_{\beta_1}(q^2_0)<0,$

where $ C^1_\beta(q)=\frac{ \mathbb{E} [ \cosh (\sqrt{\xi'(q)} g) \log \cosh (\sqrt{\xi'(q)} g) ] }{\mathbb{E} [ \cosh (\sqrt{\xi'(q)} g) ] }$ and  $C^2_\beta(q)=\frac12 \xi'(q) +\frac12 \theta(q).$
\end{enumerate}
\item there exists $0 < q_2< 1$ such that $C_{\beta_2}(q_2)>0.$
\end{enumerate}
\end{proposition}

\begin{remark}
In order to approximate $\beta^p_1$ better, we can choose 
 $\beta_1$ and $q^1_0,q^2_0,q^1_1,q^2_1$ satisfying (a) in (1) such that  $C^1_{\beta_1}(q^1_1) - C^2_{\beta_1}(q^2_0)$ is closer to 0.
We can also choose $\beta_2$ such that $\max_{q \in (0,1)} C_{\beta_2}(q)$ is smaller.
\end{remark}

We present the following example as an application of Proposition \ref{criterion}.
\begin{example}
\begin{enumerate}
\item Set $p=3.$  For $\beta_1=1.05$, choose $q^1_0=0.733$, $q^2_0=0.735$,$q^1_1=0.739,$ and $q^2_1=0.740,$ the two conditions (a) and (b) in Proposition \ref{criterion}(1) are then satisfied. For $\beta_2=1.1$ and $q_2=0.9,$ we obtain that $C_{\beta_2}(q_2)>0.$ By Proposition \ref{criterion}, the range of $\beta^3_1$ is then $[1.05, 1.1].$
\item Set $p=20.$ For $\beta_1=1.15$, choose $q^1_0=0.9999992$, $q^2_0=0.9999994$, $q^1_1=0.9999997,$ and $q^2_1=0.9999999,$ the two conditions (a) and (b) in Proposition \ref{criterion}(1) are then satisfied. For $\beta_2=1.2$ and $q_2=0.99,$ we obtain that $C_{\beta_2}(q_2)>0.$ By Proposition \ref{criterion}, the range of $\beta^p_1$ is then $[1.15, 1.2].$
\end{enumerate}
\end{example}

\section{Characterization of the Parisi measure $\mu_P$}\label{section4}
In order to prove our main results in Theorem \ref{mainthm}, we  introduce  the Parisi functional and characterize the Parisi measure in more details.
We first  present the well-known Parisi functional $\mathcal{P} (\mu)$ as follows. For any $\mu \in M[0,1]$, $\mathcal{P}(\mu)$ is defined as  
\begin{eqnarray} \label{Parisifunctional}	
\mathcal{P}(\mu)= \log 2 + \Phi_\mu (0,0)-\frac12 \int^1_0 \alpha_\mu(s) s \xi''(s) ds,
\end{eqnarray}
where $\Phi_\mu$ is the solution to the Parisi PDE on $[0,1] \times \mathbb{R}$
\begin{equation}\label{ParisiPDE}  \left\{
\begin{array}{lcl}
\partial_u \Phi_\mu(x,u)=-\frac{\xi''(u)}{2} \Big[ \partial_{xx} \Phi_\mu(x,u) +\alpha_\mu(u) \big( \partial_x  \Phi_\mu(x,u)  \big)^2 \Big] .   \\
\Phi_\mu(x,1)=\log \cosh x .
\end{array} \right. \end{equation} and $\alpha_\mu$ is the distribution function of $\mu \in M[0,1].$

Now in order to prove our main results, we will need a criterion to characterize the structure of the Parisi measure $\mu_P.$ For any $\mu \in M[0,1],$ we first define
\begin{eqnarray*}
\Gamma_\mu(u)=\mathbb{E} \big( \partial_x \Phi_\mu( B(\xi'(u)),u) \big)^2 \exp W_\mu(u)
\end{eqnarray*}
where $B=(B(t))_{t \geq 0}$ is a standard Brownian motion and 
\begin{eqnarray*}
W_\mu(u)=\int^u_0 \big( \Phi_\mu( B(\xi'(u))  ,u) - \Phi_\mu( B(\xi'(s))  ,s) \big)d \mu(s), u \in [0,1].
\end{eqnarray*}

We then consider the following two functions
\begin{eqnarray}\label{criterionfunc1}
F_\mu(u)= \Gamma_\mu(u)-u
\end{eqnarray}
and 
\begin{eqnarray}\label{criterionfunc2}
f_\mu(u)= \int^u_0 \frac{\xi''(s)}{2}  \cdot F_\mu(s) ds.
\end{eqnarray}

We  then prove the following necessary and sufficient criterion for $\mu \in M[0,1]$ to be the Parisi measure of the Ising pure $p$-spin glass models without external fields. A similar criterion was obtained by Talagrand\cite{talagrand2006parisi-b} and later utilized  by Auffinger and Chen \cite{AChen15PTRF}.
\begin{theorem}\label{thmcriterion}
$\mu_P$ is the Parisi measure if and only if the following two conditions are satisfied:
\begin{enumerate}
\item  $ f_{\mu_P}(u) \leq 0$, for $0 \leq u \leq 1.$
\item $\mu_P( \{ u \in [0,1]:f_{\mu_P}(u)=0 \})=1.$
\end{enumerate}
\end{theorem}

\begin{proof}[Proof of Theorem \ref{thmcriterion}]
Fix $\mu_0 \in M[0,1]$ and consider any $\mu \in M[0,1]$. For $\theta \in [0,1],$ define $\mu_\theta=(1-\theta)\mu_0+\theta \mu.$  By a similar computation to \cite[Lemma 3.7]{{talagrand2006parisi-b}}, we obtain that
\begin{eqnarray*}
\frac{d}{d \theta}\big\{ \mathcal{P}(\mu_\theta) \big\}_{ |_{\theta =0} }&=& \frac12 \int ^1_0 \xi''(s) F_{\mu_0}(s) \big( \alpha_\mu(s)- \alpha_{\mu_0}(s)\big) ds.
\end{eqnarray*}
Here for any $\nu \in M[0,1]$, $\alpha_\nu$ is its distribution, i.e. $\alpha_\nu (s)=\nu([0,s]), \forall s \in [0,1].$

Note that the following inequality
\begin{eqnarray}\label{dP}
\frac{d}{d \theta}\big\{ \mathcal{P}(\mu_\theta) \big\}_{ |_{\theta =0} } \geq 0,
\end{eqnarray}
is equivalent to
\begin{eqnarray*}
0 &\leq & \frac12 \int ^1_0 \xi''(s) F_{\mu_0}(s)  \alpha_\mu(s) ds-\frac12 \int ^1_0 \xi''(s) F_{\mu_0}(s)   \alpha_{\mu_0}(s) ds\\
&=&  \frac12 \int ^1_0 \int^s_0 \xi''(s) F_{\mu_0}(s) d\mu(u)  ds-  \frac12 \int ^1_0 \int^s_0 \xi''(s) F_{\mu_0}(s)  d\mu_0(u)  ds\\
&=&  \frac12 \int ^1_0 \int^1_u \xi''(s) F_{\mu_0}(s) ds d\mu(u)  -\frac12 \int ^1_0 \int^1_u \xi''(s) F_{\mu_0}(s) ds  d\mu_0(u)  
\end{eqnarray*}
and then
\begin{eqnarray*}
 \int ^1_0 \int^1_u \frac{\xi''(s)}{2} \cdot F_{\mu_0}(s) ds d\mu(u)  \geq  \int ^1_0 \int^1_u \frac{\xi''(s)}{2} \cdot F_{\mu_0}(s) ds  d\mu_0(u)  .  
\end{eqnarray*}
Therefore, it yields that \eqref{dP} is equivalent to 
\begin{eqnarray}\label{Condition1}
\int ^1_0 f_{\mu_0}(u)  d\mu(u)  \leq  \int ^1_0 f_{\mu_0}(u)   d\mu_0(u).  
\end{eqnarray}
Therefore the statement that \eqref{dP} holds for any $\mu \in M[0,1]$  is equivalent to that
\begin{eqnarray*}
 \mu_0\big( \big \{ u \in [0,1] |  f_{\mu_0}(u) =\sup_{t \in [0,1]}  f_{\mu_0}(t) \big\} \big)=1.
\end{eqnarray*}

If  $\mu_0$ is the Parisi measure $\mu_P$ that minimizes the Parisi functional $\mathcal{P},$ then \eqref{dP} holds  for any $\mu \in M[0,1]$.
 It was shown   \cite[Theorem 1]{AChen15PTRF}, that in the absence of external field, the support of any Parisi measure contains the origin. It then yields that 
\begin{eqnarray*}
0=f_{\mu_P}(0) =\sup_{t \in [0,1]}  f_{\mu_P}(t).
\end{eqnarray*}
Then the statement that  \eqref{dP} holds for any $\mu \in M[0,1],$ is equivalent to say that
\begin{enumerate}
\item $f_{\mu_0}(u)\leq 0, \forall u \in [0,1].$
\item $f_{\mu_0}(u)= 0$, for any $u$ in the support of $\mu_0.$
\end{enumerate}

Now we assume that  \eqref{dP} holds for any $\mu \in M[0,1].$ For any $\epsilon > 0$, there exists some $\delta > 0$ such that $\mathcal{P}(\mu_\theta) -\mathcal{P}(\mu_0) \geq -\epsilon \theta$ for any $0 < \theta < \delta.$ The convexity of $\mathcal{P}$ yields that
$$\mathcal{P}(\mu_\theta)\leq (1-\theta) \mathcal{P}(\mu_0)+\theta \mathcal{P}(\mu),$$
and then 
$$(1-\theta) \mathcal{P}(\mu_0)+\theta \mathcal{P}(\mu) \geq \mathcal{P}(\mu_0)-\epsilon \theta.$$
Therefore, we obtain that
$$\theta \big( \mathcal{P}(\mu)-\mathcal{P}(\mu_0) \big)=\theta  \mathcal{P}(\mu)+(1-\theta)\mathcal{P}(\mu_0) - \mathcal{P}(\mu_0)\geq-\epsilon \theta.$$
Since $\mathcal{P}(\mu) \geq \mathcal{P}(\mu_0) -\epsilon$ for all $\epsilon>0$ and $\mu \in M[0,1],$ $\mu_0$ is the unique minimizer of the Parisi functional, i.e. $\mu_0=\mu_P.$

\end{proof}

\section{Proof of Theorem \ref{mainthm}}\label{section3}

In this section, we will prove our main results in Theorem \ref{mainthm}. We establish the first part of Theorem \ref{mainthm} in Section \ref{section3.1} and the second part in Section \ref{section3.2}.

\subsection{Proof of Replica Symmetry}\label{section3.1}
In this section, we will prove that the Parisi measure is replica symmetric only at high temperature. To be more specific, we show that  $\mu_P=\delta_0$ if and only if $0 \leq \beta \leq \beta^p_1$.


Recall the two functions $C_\beta(u)$ and $D_\beta(u)$ in \eqref{boundary}. As a corollary of Theorem \ref{thmcriterion}, we introduce the following criterion for $\mu_P$ to be replica symmetric:
\begin{proposition}[RS criterion]\label{rscriterion}
The Parisi measure $\mu_P$ is replica symmetric if and only if for  $q \in [0,1]$ with $D_\beta(q)=0$, it holds that $C_\beta(q) \leq 0.$

\end{proposition}

 We introduce the properties of $C_\beta(u)$ and $D_\beta(u)$ as follows and defer their proof to the end of this section:
\begin{theorem}\label{propcd1}
\begin{enumerate}
 \item $C_\beta(u)$ has either zero or two critical points in $(0,+\infty)$ and $D_\beta(u)$ has the same sign as $\frac{d}{du} \big\{ C_\beta(u)\big\}$. Moreover, if $C_\beta(u)$ has two critical points  $u_2>u_1>0,$ then $u_1$ is a local minimum point and $u_2$ is a local maximum point.
 \item For any $\beta>0$, $\lim_{u \rightarrow +\infty} C_\beta(u) = -\infty$.  Also, $ \lim_{\beta \rightarrow + \infty} C_\beta(1) = +\infty.$ 
 \item For any $\beta>0,$ $C_\beta(0)=0$ and there exists $\epsilon>0$ such that $C_\beta(u)$ is strictly decreasing in $(0,\epsilon).$
 \item $ D_0(u)<0,$ for any $u >0$ and  $D_\beta(1)<0$.
 \end{enumerate}
\end{theorem}

For any $u \geq 0,$ we define $$T(u):= \frac{u \mathbb{E} [ \cosh (\sqrt{\xi'(u)} g) ] }{ \mathbb{E} [ \tanh^2 (\sqrt{\xi'(u)} g)  \cosh (\sqrt{\xi'(u)} g) ] } -1.$$ 
The following property of $T(u)$ is pivotal to the proof of Theorem \ref{propcd1}. We will prove it at the end of this section.
\begin{lemma}\label{tu}
$T(u)$ is a convex function in $(0,+\infty).$
\end{lemma}

Now we consider the collection of all probability measures on $[0, 1]$ whose support consists of finitely many atoms, denoted by $M_d[0,1]$. For each $\mu \in M_d[0,1],$ there exists the following unique triplet $(k,\mathbf{m},\mathbf{q})$ satisfying that
\begin{eqnarray*}
&\mathbf{m}:& m_0=0 \leq m_1 < m_2 < \cdots <m_k \leq m_{k+1}=1, \\
&\mathbf{q}:& q_0=0 \leq q_1 < q_2 < \cdots < q_{k+1} \leq q_{k+2}=1.
\end{eqnarray*}
such that $\mu([0,q_p])=m_p$ for $0 \leq p \leq k+1.$ Since the distribution of $\mu$ is a step function, we can then solve the Parisi PDE \eqref{ParisiPDE} explicitly by the Cole-Hopf transformation. To be more specific, for $q_{k+1} \leq u \leq 1,$
\begin{eqnarray*}
\Phi_\mu(x,u)= \log \cosh x+\frac12 \big[ \xi'(1)-\xi'(u) \big]
\end{eqnarray*}
and for $q_p \leq u <q_{p+1}$ with $0 \leq p \leq k,$
\begin{eqnarray*}
\Phi_\mu(x,u)= \frac{1}{m_p} \log \mathbb{E} \exp m_p \Phi_\mu(x+g \sqrt{\xi'(q_{p+1})-\xi'(u)} ,q_{p+1}),
\end{eqnarray*}
where $g$ is a standard Gaussian random variable.

Based on the explicit solution of the Parisi PDE introduced above, we are now ready to prove Proposition \ref{rscriterion}.
\begin{proof}[Proof of Proposition \ref{rscriterion}]

Consider any $\mu \in M_d[0,1]$ whose support consists of only one point. Then $\mu$ corresponds to the following two sequences:
\begin{eqnarray*}
&\mathbf{m}:& 0=m_0 \leq m_1 =1, \\
&\mathbf{q}:& 0=q_0 = q_1 < q_2 =1,
\end{eqnarray*} 
and $\mu([0,q_p])=m_p$ for $p=0,1.$

Based on the Cole-Hopf transformation, it holds that for $0 \leq u \leq 1,$
\begin{eqnarray*}
\Phi_\mu(x,u)= \log \cosh x +\frac12 \big[ \xi'(1)-\xi'(u) \big],
\end{eqnarray*}
which implies that
\begin{eqnarray*}
W_\mu(u)= \log \cosh B(\xi'(u)), u \in [0,1],
\end{eqnarray*}
and
\begin{eqnarray*}
\Gamma_\mu(u)=\frac{\mathbb{E}_g \tanh^2 \big( \sqrt{\xi'(u)} g\big) \cosh \big( \sqrt{\xi'(u)} g\big)}{ \mathbb{E}_g \cosh \big( \sqrt{\xi'(u)} g\big) }.
\end{eqnarray*}
Here $g$ is a standard Gaussian random variable and we use the relation $\mathbb{E}_g \cosh \big( \sqrt{\xi'(u)} g\big)= \exp\big( \frac{\xi'(u)}{2}\big).$
Note that $D_\beta(u) =\Gamma_\mu(u)-u=F_{\mu}(u).$

Recall the definition of \eqref{boundary} and note that
\begin{eqnarray}\label{compderiv}
\frac{d}{du} \big\{ C_\beta(u) \big\}&=&\frac{\xi''(u)}{2\sqrt{\xi'(u)}} \cdot \frac{    \Big[ \mathbb{E} [ g \sinh (\sqrt{\xi'(u)} g) \log \cosh (\sqrt{\xi'(u)} g) ]+\mathbb{E} [g \sinh (\sqrt{\xi'(u)} g)  ] \Big] }{\mathbb{E}[ \cosh \big( \sqrt{\xi'(u)} g\big)]} \nonumber \\
&&- \frac{ \xi''(u)}{2} \cdot \frac{ \mathbb{E} [ \cosh (\sqrt{\xi'(q)} g) \log \cosh (\sqrt{\xi'(q)} g) ]  }{\mathbb{E}[ \cosh \big( \sqrt{\xi'(u)} g\big)]}-\frac12 \xi''(u)-\frac12 u\xi''(u) \nonumber \\
&=&\frac{\xi''(u)}{2} \cdot \frac{    \mathbb{E} [  \tanh (\sqrt{\xi'(u)} g) \sinh (\sqrt{\xi'(u)} g) ]  }{\mathbb{E}[ \cosh \big( \sqrt{\xi'(u)} g\big)]} -\frac12 u\xi''(u) \nonumber \\
&=&\frac{\xi''(u)}{2} \cdot F_{\mu}(u)= \frac{\xi''(u)}{2} \cdot D_\beta(u) \nonumber \\
&=&  \frac{d}{d u} \big \{ f_{\mu}(u) \big \}, 
\end{eqnarray}
where the second equality follows from a standard application of  Gaussian integration by parts.
Since $C_\beta(0)=f_{\mu}(0),$ we then obtain that $C_\beta(u)=f_{\mu}(u).$ Also for $u \geq 0,$ $\frac{d}{d u} \big \{ f_{\mu}(u) \big \}$ has the same sign as  $D_\beta(u).$

By Theorem \ref{thmcriterion}, the Parisi measure $\mu_P$ is RS if and only if  $ f_{\mu_P}(u) \leq 0$, for $0 \leq u \leq 1.$ Based on the definition of  $ f_{\mu_P}(u)$ and  $ F_{\mu_P}(u)$, $\mu_P$ is RS if and only if for  $q \in [0,1]$ with $F_{\mu_P}(q)=0$, it holds that $f_{\mu_P}(q) \leq 0.$ Equivalently, $\mu_P$ is RS if and only if for  $q \in [0,1]$ with $D_\beta(q)=0$, it holds that $C_\beta(q) \leq 0.$

\end{proof}

We now prove the first part of our main results in Theorem \ref{mainthm}.
\begin{proof}[Proof of Theorem \ref{mainthm}(1)]
For convenience, we set $Y=\sqrt{\xi'(u)}$ and then $\frac{d}{du}Y=\frac{p-1}{2u}Y.$ Now we recall the system of equations \eqref{boundary} and rewrite it as follows:
\begin{eqnarray*}
\left\{
\begin{array}{lcl}
C_\beta(u)=\frac{ \mathbb{E} [ \cosh (Y g) \log \cosh (Y g) ] }{\mathbb{E} [ \cosh (Y g) ] }-\frac12 Y^2 -\frac12 \theta(u)=0 \\
D_\beta(u)=\frac{ \mathbb{E} [ \tanh^2 (Y g)  \cosh (Y g) ] }{\mathbb{E} [ \cosh (Y g) ] }- u=0
\end{array} \right..
\end{eqnarray*}

 We first prove that if the solution to \eqref{boundary} exists, then it must be unique. Assume that $(\beta_0,u_0)$ is a solution to \eqref{boundary}. To prove the uniqueness, it suffices for us to show that $\frac{d}{d\beta} \big\{C_\beta(u) \big\}_{\big|(\beta,u)=(\beta_0,u_0)}>0.$ 

Define the following function with respect to $Y=\sqrt{\xi'(u)}$ $$C_1(Y):=\frac{ \mathbb{E} [ \cosh (Y g) \log \cosh (Y g) ] }{\mathbb{E} [ \cosh (Y g) ] }-\frac12 Y^2,$$  and then $C_\beta(u)=C_1(Y) -\frac12 [u \xi'(u) -\xi(u)]$. 
Recalling from \eqref{compderiv}, we obtain that 
\begin{eqnarray*}
 \frac{\xi''(u)}{2} \cdot D_\beta(u)=\frac{d}{du} \big\{ C_\beta(u) \big\} = \frac{\partial}{\partial Y} \big\{ C_1(Y) \big\} \cdot \frac{(p-1)}{2u}Y -\frac12 u \xi''(u) ,
\end{eqnarray*}
 which implies that
 \begin{eqnarray*}
 \frac{\partial}{\partial Y} \big\{ C_1(Y) \big\}_{\big|(\beta,u)=(\beta_0,u_0)}=u\cdot Y.
 \end{eqnarray*}
 We now compute $\frac{\partial}{\partial \beta} \big\{C_\beta(u) \big\}_{\big|(\beta,u)=(\beta_0,u_0)}$ as follows:
 \begin{eqnarray}\label{cpartial}
 &&\frac{\partial}{\partial \beta} \big\{C_\beta(u) \big\}_{\big|(\beta,u)=(\beta_0,u_0)} \nonumber \\
 &=&  \Big\{  uY \cdot \Big( \frac{(p-1)Y}{2 u}\cdot \frac{\partial u}{\partial \beta} +  \frac{ Y}{ \beta}  \Big)- \frac{1}{2}u \xi''(u) \cdot \frac{\partial u}{\partial \beta}  -\frac1\beta [u \xi'(u) -\xi(u)] \Big\}_{\big|(\beta,u)=(\beta_0,u_0)} \nonumber \\
 &=&  \Big\{  \Big( \frac{(p-1)Y^2}{2 } - \frac{1}{2}u \xi''(u)  \Big) \cdot \frac{\partial u}{\partial \beta} +  \frac{ uY^2}{ \beta}  -\frac1\beta [u \xi'(u) -\xi(u)] \Big\} _{\big|(\beta,u)=(\beta_0,u_0)} \nonumber \\
 &=& \frac{1}{\beta} \xi(u_0)>0,
 \end{eqnarray}
which guarantees the uniqueness of the solution to \eqref{boundary}.

We then prove that for any $p \geq 3,$ \eqref{boundary} must have a solution $(\beta,u)$ with $\beta>0$ and $u \in (0,1).$  

We show that when $\beta$ is sufficiently large, $C_\beta(u)$ has two critical points in $(0,1)$.
By Theorem \ref{propcd1}(2), we know that  for $\beta$ sufficiently large, it holds that $C_\beta(1)>0.$ We then claim that $C_\beta(u)$ has 2 critical points $u_1<u_2$ in $(0,+\infty).$ Indeed, by Theorem \ref{propcd1}(1), $C_\beta(u)$ has either zero or two critical points in $(0,+\infty).$ Assume that $C_\beta(u)$ has no critical points in $(0,+\infty).$ Since by Theorem \ref{propcd1}(1)(2), it holds that $\lim_{u \rightarrow +\infty} C_\beta(u) = -\infty$ and then  it holds that $C_\beta(u)<0$ for any $u>0$, which leads to a contradiction.    
Moreover, by Theorem \ref{propcd1}(1)(4), we obtain that $D_\beta(1)<0$ and then $\frac{d}{du} \big\{ C_\beta(u)\big\}_{|u=1}<0.$ Since $u_1$ is a local minimum point of $C_\beta(u)$ and $u_2$ is a local maximum point, we conclude that $0<u_1<u_2<1$ and $C_\beta(u_2)>0$ for $\beta$ sufficiently large.

We then show that when $\beta$ is sufficiently small, $C_\beta(u)$ has no critical points in $(0,1)$.
Since by Theorem \ref{propcd1}(2), for any $u >0,$ $ D_0(u)<0,$ it yields that for $\beta$ sufficiently small, $D_\beta(u)$ has no zeros in $(0,+\infty)$ and $C_\beta(u)$ has no critical points in $(0,+\infty)$. By Theorem \ref{propcd1}(3), $C_\beta(u)$ is strictly decreasing near zero.  The fact that $C_\beta(0)=0$ then yields that when $\beta$ sufficiently small, $C_\beta(u)<0$ for any $u>0$.

Combining the arguments above, for any $p \geq 3,$ by the continuity of $\beta$, there must exist a solution to \eqref{boundary} in $(0,\infty) \times (0,\infty)$, which is then also unique. We denote this unique solution to \eqref{boundary} by $(\beta^p_1,u^p_1).$ Now we claim that it holds that $u^p_1 \in (0,1).$ Indeed, if $u^p_1 \geq 1,$ by continuity, there must exist $\beta> \beta^p_1,$ such that $D_\beta(1)>0,$ which leads to contradiction. 

Since for $\beta$ sufficiently small, it holds that $C_\beta(u)<0,u>0$, we then obtain that for $\beta \in (0,\beta^p_1),$ $C_\beta(u)<0,u>0.$ By Proposition \ref{rscriterion}, the Parisi measure is replica symmetric for $\beta \in (0,\beta^p_1).$ Since $\frac{d}{d\beta} \big\{C_\beta(u) \big\}_{\big|(\beta,u)=(\beta^p_1,u^p_1)}>0,$  the Parisi measure is not replica symmetric for $\beta \in (\beta^p_1,+\infty).$

\end{proof}

We now prove Theorem \ref{propcd1} as follows:
\begin{proof}[Proof of Theorem \ref{propcd1}]
Firstly, we prove Theorem \ref{propcd1}(1).
 Recalling that $Y=\sqrt{\xi'(u)}$, we rewrite $T(u)$ as follows $$T(u)= \frac{u \mathbb{E} [ \cosh (Y g) ] }{ \mathbb{E} [ \tanh^2 (Y g)  \cosh (Y g) ] } -1\text{ for } u \geq 0.$$ Note that  for $u \geq 0,$ $T(u)$ has the opposite sign as $\frac{d}{du} \big\{ C_\beta(u) \big\}$ and $D_\beta(u).$ In particular, $D_\beta(u)=0$ is then equivalent to $T(u)=0.$

When $u \rightarrow 0^+,$  it holds that, by L'Hopital's rule, 
\begin{eqnarray*}
 \lim_{u \rightarrow 0^+} \frac{u \mathbb{E} [ \cosh (Y g) ] }{ \mathbb{E} [ \tanh^2 (Y g)  \cosh (Y g) ] } &=&  \lim_{u \rightarrow 0^+} \frac{ \mathbb{E} [ \cosh (Y g) ] + \frac{p-1}{2}Y \mathbb{E} [ \cosh (Y g) ]}{\frac12 \xi''(u) \mathbb{E} [ \tanh^2 (Y g)  \cosh^{-1}(Y g)+\cosh^{-3}(Yg) +\cosh(Yg) ] }  \\
 &=&+\infty,
 \end{eqnarray*}
which yields that $\lim_{u \rightarrow 0^+} T(u)=+\infty.$

Moreover, since the following inequality holds for any $u>0,$
$$\mathbb{E} [ \tanh^2 (Y g)  \cosh (Y g) ]< \mathbb{E} [  \cosh (Y g) ],$$
we then obtain that 
$$T(u)>u-1 \text{ and then } \lim_{u \rightarrow +\infty}T(u)=+ \infty.$$

In order to prove Theorem \ref{propcd1}, the key observation is that
 $T(u)$ is a convex function in $(0,+\infty)$ given by Lemma \ref{tu}. Therefore, $T(u)$ has either 2 or 0 zeros in $(0,+\infty).$ When $T(u)$ has 2 zeros, we denote them by $u_1<u_2.$  Since $\lim_{u \rightarrow 0^+} T(u)= +\infty$ and  $\lim_{u \rightarrow +\infty}T(u)=+ \infty$, there exists $\epsilon>0$ such that $T(u)>0$ for $(u_1-\epsilon,u_1) \cup (u_2,u_2+\epsilon)$ and 
$T(u)<0$ for $(u_1,u_1+\epsilon) \cup ( u_2-\epsilon,u_2).$ Since for $u \geq 0,$ $T(u)$ has the opposite sign as $\frac{d}{du} \big\{ C_\beta(u) \big\}$, we obtain that $u_1$ is a local minimum point of $C_\beta(u)$ and $u_2$ is a local maximum point. 

Now we prove Theorem \ref{propcd1}(2). We find an upper bound for  $\mathbb{E}[\cosh(Yg) \log \cosh(Yg)]$ as follows:
\begin{eqnarray*}
\mathbb{E}[\cosh(Yg) \log \cosh(Yg)]&=& \frac{2}{\sqrt{2 \pi}} \int^{+\infty}_0 e^{-\frac12s^2} \cosh(Ys) \log \cosh(Ys) ds\\
&\leq & \frac{2}{\sqrt{2 \pi}} \int^{+\infty}_0 e^{-\frac12s^2} \cosh(Ys) \cdot (Ys) ds\\
&=& \frac{2Y}{\sqrt{2 \pi}} +\frac{2Y^2}{\sqrt{2 \pi}}  \int^{+\infty}_0 e^{-\frac12s^2} \sinh(Ys) ds\\
&=& \frac{2Y}{\sqrt{2 \pi}} +  {e^{\frac{Y^2}{2}} \cdot {Y^2}  \text{erf }(\frac{Y}{\sqrt{2}})},
\end{eqnarray*}
where $\text{erf }(u)=\frac{2}{\sqrt{2\pi}} \int^u_0 e^{-\frac{1}{2}t^2}dt,$ for $u \geq 0.$

We then obtain that for $u \geq 0,$
\begin{eqnarray*}
C_\beta(u) &\leq& \sqrt{\frac{2}{ \pi}} \cdot  Ye^{-\frac{Y^2}{2}} +  {\text{erf }(\frac{Y}{\sqrt{2}})} \cdot \xi'(u) -\frac12 \xi'(u) -\frac12 [u \xi'(u) -\xi(u)]\\
&=& \sqrt{\frac{2}{ \pi}} \cdot  Ye^{-\frac{Y^2}{2}} +  \Big( {\text{erf }(\frac{Y}{\sqrt{2}})} - \frac12 \Big) \beta^2pu^{p-1}  -\frac{\beta^2(p-1)}{2} \cdot  u^p
\end{eqnarray*}
which implies that  $\lim_{u \rightarrow +\infty} C_\beta(u) = -\infty$. 

We also find the following lower bound for  $\mathbb{E}[\cosh(Yg) \log \cosh(Yg)]$:
\begin{eqnarray*}
\mathbb{E}[\cosh(Yg) \log \cosh(Yg)]
&\geq & \frac{2}{\sqrt{2 \pi}} \int^{+\infty}_0 e^{-\frac12s^2} \cosh(Ys) \cdot (Ys) ds- \log 2 \cdot  \mathbb{E}[\cosh(Yg)]\\
&=& \frac{2Y}{\sqrt{2 \pi}} +  {e^{\frac{Y^2}{2}} \cdot {Y^2}  \text{erf }(\frac{Y}{\sqrt{2}})} - \log 2 \cdot  \mathbb{E}[\cosh(Yg)],
\end{eqnarray*}
which implies that
\begin{eqnarray*}
C_\beta(1) &\geq& - \log 2 +\sqrt{\frac{2}{ \pi}} \cdot  Ye^{-\frac{Y^2}{2}} + p \cdot \Big(  {\text{erf }(\frac{Y}{\sqrt{2}})}  -1 \Big) \cdot \beta^2  +\frac{1}{2}  \beta^2 \end{eqnarray*}
and then  $ \lim_{\beta \rightarrow + \infty} C_\beta(1) = +\infty.$

We then prove Theorem \ref{propcd1}(3).  By Theorem \ref{propcd1}(1), $C_\beta(u)$ has either zero or two critical points in $(0,+\infty)$. If $C_\beta(u)$ has no critical points in $(0,+\infty),$ since $\lim_{u \rightarrow +\infty} C_\beta(u) = -\infty$, $C_\beta(u)$ is then strictly decreasing in $(0,+\infty).$ We then assume $C_\beta(u)$ has two critical points in $(0,+\infty).$ If $C_\beta(u)$ is increasing near 0, the number of critical points  in $(0,+\infty)$ will then be odd, which leads to a contradiction.

Next, we prove Theorem \ref{propcd1}(4).
For any $q>0$, it holds that for $\beta=0$,
\begin{eqnarray*}
D_0(q)=\frac{ \mathbb{E} [ \tanh^2 (Y g)  \cosh (Y g) ] }{ \mathbb{E} [ \cosh (Y g) ] }-q=-q<0.
\end{eqnarray*}

Note that  for $\beta,u>0,$
$ \mathbb{E} [ \tanh^2 (Y g)  \cosh (Y g) ]< { \mathbb{E} [ \cosh (Y g) ] } $,
it then yields that $D_\beta(1)<0$.

\end{proof}

Based on Proposition \ref{rscriterion} and Theorem \ref{propcd1}, we prove Proposition \ref{criterion} as follows:
\begin{proof}[Proof of Proposition \ref{criterion}]
If $C_\beta(u)$ has two critical points in $(0,1),$ we denote the larger one by $c_\beta$, which is a local maximum point.
We claim that Condition (1) in Proposition \ref{criterion} is equivalent to $C_{\beta_1}(c_{\beta_1})<0$. Also, Condition (2) in Proposition \ref{criterion} is equivalent to $C_{\beta_1}(c_{\beta_2})>0$. Then by intermediate value theorem, there exists $\beta^p_1 \in (\beta_1,\beta_2)$ such that $C_{\beta^p_1}(c_{\beta^p_1})>0$. Note that $D_\beta(c_\beta)=0$ for any $\beta>0.$  Set $q^p_1=c_{\beta^p_1}.$ Then $(\beta^p_1,q^p_1)$ is the unique solution to \eqref{boundary}. 

In order to prove the claim, we first consider Condition (1) in Proposition \ref{criterion}. Since $D_{\beta_1}(q^1_0) > D_{\beta_1}(q^2_0)>0>D_{\beta_1}(q^1_1) > D_{\beta_1}(q^2_1)$, there exists $c_{\beta_1} \in (q^2_0, q^1_1 )$ such that $D_{\beta_1}(c_{\beta_1})=0$ and $D'_{\beta_1}(c_{\beta_1})<0$. Therefore by Theorem \ref{propcd1}, $q=c_{\beta_1}$ is a local maximum of $C_{\beta_1}(q).$ 

Note that for $q \geq 0,$ $$\frac{d}{dq} \big\{C^1_\beta(q)\big\}=\frac{\xi''(q)}{2} \cdot \Big( \frac{ \mathbb{E} [ \sinh^2 (\sqrt{\xi'(q)} g)  \cosh^{-1} (\sqrt{\xi'(q)} g) ] }{\mathbb{E} [ \cosh (\sqrt{\xi'(q)} g) ] }+1 \Big)>0,$$
which implies that $C^1_{\beta_1}(c_{\beta_1})<C^1_{\beta_1}(q^1_1).$
We then obtain the following relation
\begin{eqnarray*}
C_{\beta_1}(c_{\beta_1})= C^1_{\beta_1}(c_{\beta_1}) - C^2_{\beta_1}(c_{\beta_1}) < C^1_{\beta_1}(q^1_1) - C^2_{\beta_1}(q^2_0)<0.
\end{eqnarray*}

We then consider  Condition (2) in Proposition \ref{criterion}.  By Theorem \ref{propcd1}, $C_{\beta_2}(u)$ has two critical points in $(0,1)$ and $c_{\beta_2}$  is the local maximum point. Since $C_{\beta_2}(q_2)>0$, it yields that $C_{\beta_2}(c_{\beta_1})>0.$

\end{proof}

Finally, we prove Lemma \ref{tu} as follows:
\begin{proof}[Proof of Lemma \ref{tu}]
We start by computing the derivative of $\frac{d}{du}\big\{ T(u) \big\}$ as follows:
\begin{eqnarray}\label{similarcomp}
\frac{d}{du}\big\{ T(u) \big\}&=& \frac{1}{ \big( \mathbb{E} [ \tanh (Y g)  \sinh (Y g) ] \big)^2} \cdot \Big\{  \nonumber \\
&& \mathbb{E} [ \tanh (Y g)  \sinh (Y g) ] \cdot \Big( \mathbb{E} [ \cosh (Y g) ] + \frac{p-1}{2} Y \mathbb{E} [ g  \sinh (Y g) ] \Big) \nonumber\\
&&- \frac{p-1}{2}Y \mathbb{E} [   \cosh (Y g) ] \cdot \Big( \mathbb{E} [g \cosh^{-2} (Y g)  \sinh (Y g) ] +\mathbb{E} [ g\tanh (Y g)  \cosh (Y g) ] \Big) \Big\}\nonumber\\
&=&\frac{1}{ \big( \mathbb{E} [ \tanh (Y g)  \sinh (Y g) ] \big)^2} \cdot \Big\{ (p-1)Y^2 \mathbb{E} [   \cosh (Y g) ]   \mathbb{E} [ \cosh^{-3} (Y g)  \sinh^2 (Y g) ] \nonumber \\
&& \mathbb{E} [ \tanh (Y g)  \sinh (Y g) ]   \mathbb{E} [ \cosh (Y g) ] - (p-1) Y^2 \mathbb{E} [   \cosh (Y g) ] \mathbb{E} [ \cosh^{-1} (Y g) ]  \Big\} \nonumber \\
&=& \frac{ \mathbb{E} [  \cosh (Y g) ] \Big( \mathbb{E} [ \tanh (Y g)  \sinh (Y g) ] -(p-1) Y^2 \mathbb{E} [ \cosh^{-3} (Y g) ]  \Big) }{  \big( \mathbb{E} [ \tanh (Y g)  \sinh (Y g) ] \big)^2 }
\end{eqnarray}
Here the second equality follows from another standard application of Gaussian integration by parts.

We then start computing the second derivative of $T(u).$ In order to simplify our computations, we set $a_k:=\mathbb{E}[\cosh^k(Yg)]$, for any $k \in \mathbb{Z}.$  We can then rewrite $\frac{d}{du}\big\{ T(u) \big\}$ as follows:
\begin{eqnarray*}
\frac{d}{du}\big\{ T(u) \big\}= \frac{  a_1    }{  \big( a_1-a_{-1}\big) } -(p-1) Y^2 \cdot \frac{ a_1 a_{-3}   }{  \big( a_1-a_{-1}\big)^2 }
\end{eqnarray*}
Here we use the relation that $\mathbb{E} [ \tanh (Y g)  \sinh (Y g) ]= \mathbb{E} [ \cosh (Y g) ]-\mathbb{E} [   \cosh^{-1} (Y g) ].$

Note that for any $k \in \Z$,
\begin{eqnarray*}
\frac{d}{dY} \big\{ a_k \big\} &=& k \mathbb{E} [g \cosh^{k-1}(Yg) \sinh(Yg)] \\
&=&k Y a_k+ k(k-1)Y \mathbb{E} [\cosh^{k-2}(Yg) \sinh^2(Yg)]\\
&=& k^2 Y a_k -k(k-1)Y a_{k-2}.
\end{eqnarray*}
In particular, it yields that $$\frac{d}{dY} \{a_1\}=Ya_1,   \frac{d}{dY} \{a_{-1} \}=Ya_{-1}-2Ya_{-3}\text{ and }\frac{d}{dY} \{a_{-3} \}=9Ya_{-3}-12Ya_{-5}.$$

We are now ready to compute the second derivative of $T(u)$ as follows:
\begin{eqnarray*}
\frac{d^2}{du^2}\big\{ T(u) \big\}&=& \frac{(p-1)Y}{(a_1-a_{-1})^22u}  \cdot \Big \{ a_1 \frac{d}{dY} \{a_{-1}\} -a_{-1} \frac{d}{dY} \{a_1\} \Big\} -\frac{(p-1)^2Y^2}{u} \cdot \frac{ a_1 a_{-3}   }{  \big( a_1-a_{-1}\big)^2 }\\
&&-\frac{(p-1)^2Y^3}{2u(a_1-a_{-1})^3} \Big\{ a^2_1 \frac{d}{dY} \{a_{-3}\}-a_1a_{-1} \frac{d}{dY} \{a_{-3}\}-a_{-1} a_{-3} \frac{d}{dY} \{a_{1} \}\\
&&- a_1a_{-3} \frac{d}{dY} \{a_{1}\} +2a_1 a_{-3} \frac{d}{dY} \{a_{-1}\}   \Big\} \\
&=& - \frac{p(p-1)Y^2a_1a_{-3}}{u(a_1-a_{-1})^2}  \\
&&-\frac{(p-1)^2Y^4a_1}{u(a_1-a_{-1})^3} \cdot \Big\{ -2a^2_{-3}-4a_{-1}a_{-3}+4a_1a_{-3}-6a_1a_{-5}+6a_{-1}a_{-5} \Big\} \\
&=&  \frac{(p-1)a_1Y^2}{u(a_1-a_{-1})^3} \Big\{ -pa_1a_{-3}+pa_{-1}a_{-3} +2(p-1) Y^2a^2_{-3} \\
&&+4(p-1)Y^2a_{-1}a_{-3} -4(p-1)Y^2a_1a_{-3}+6(p-1)Y^2a_1a_{-5} -6(p-1)Y^2a_{-1}a_{-5} \Big\}\\
&:=&  \frac{(p-1)a_1Y^2}{u(a_1-a_{-1})^3}  \cdot G_0(u)
\end{eqnarray*}

Note that both $a_1$ and $a_1-a_{-1}$ are positive for $u >0.$ In order to show that $T(u)$ is convex on $[0,\infty)$, it's then  equivalent for us to show that $G_0(u)>0$ for $u \in (0,\infty).$ 

We now rewrite $G_0(u)$ as follows:
\begin{eqnarray*}
G_0(u)&=&-pa_{-3} \mathbb{E} [\cosh(Yg) \tanh^2(Yg) ]-4(p-1)Y^2 a_{-3} \mathbb{E} [\cosh(Yg) \tanh^2(Yg)]\\
&&+2(p-1)Y^2a^2_{-3}+6(p-1)Y^2a_{-5} \mathbb{E}[\cosh(Yg) \tanh^2(Yg)] \\
&=& \frac{3(p-1)}{2}Y^2 \mathbb{E}[\cosh(Yg) \tanh^2(Yg)] \cdot (4a_{-5}-3a_{-3}) \\
&&+a_{-3} \Big\{ \frac12(p-1)Y^2 \mathbb{E}[\cosh(Yg) \tanh^2(Yg)]+2(p-1)Y^2a_{-3}-p\mathbb{E}[\cosh(Yg)\tanh^2(Yg)] \Big\}.
\end{eqnarray*}

Since for $u >0,$ $\frac{1}{Y}\mathbb{E}[g\sinh(Yg)\cosh^{-4}(Yg)]=-3a_{-3} +4a_{-5}> 0$,  it suffices for us to show that 
\begin{eqnarray}\label{ineq0}
\frac12(p-1)Y^2 \big(  \mathbb{E}[ \tanh^2(Yg) \cosh(Yg)]+4a_{-3} \big)>p\mathbb{E}[\tanh^2(Yg)\cosh(Yg)].
\end{eqnarray}
Notice that 
$$Y\mathbb{E}[\tanh^2(Yg) \cosh(Yg) +4a_{-3}]=\mathbb{E}\big[g\big(\sinh(Yg)+ 2 \sinh(Yg) \cosh^{-2}(Yg)+\arctan\big( \sinh(Yg) \big) \big) \big].$$
Therefore \eqref{ineq0} is then equivalent to 
\begin{eqnarray*}
\frac12(p-1) \mathbb{E}\big[Yg\big(\sinh(Yg)+ 2 \sinh(Yg) \cosh^{-2}(Yg)+\arctan\big( \sinh(Yg) \big) \big) \big]>p\mathbb{E}[\tanh^2(Yg)\cosh(Yg)].
\end{eqnarray*}

It's then enough for us to show that, for any $x \in \mathbb{R},$
\begin{eqnarray}\label{ineq1}
\frac12(p-1) x\big(\sinh x+ 2 \sinh x \cosh^{-2}x+\arctan ( \sinh x ) \big) >p\tanh^2x\cosh x.
\end{eqnarray}
We reformulate \eqref{ineq1} as follows:
\begin{eqnarray*}
\frac{x}{\tanh x} \geq \frac{2p }{(p-1)  \Big(1+ 2  \cosh^{-2}x+\frac{\arctan ( \sinh x )}{\sinh x} \Big)  }
\end{eqnarray*}

Based on the inequality $\frac{\arctan(t)}{t} \geq \frac{3}{1+2\sqrt{1+t^2}}$ for $t \in \mathbb{R}$ in \cite{arcsinh}, we set $t= \sinh x$. In order to prove \eqref{ineq1}, it suffices to show that for $x \geq 0,$
\begin{eqnarray}\label{ineq2}
\frac{x}{\tanh x} \geq \frac{2p}{(p-1)  \big(1+ 2  \cosh^{-2}x+ \frac{3}{1+2 \cosh x}  \big)  }.
\end{eqnarray}
Now we set $x=\text{arctanh} t$, then \eqref{ineq2} is equivalent to, for $ t \in (-1,1),$
\begin{eqnarray*}
 \frac{2p}{(p-1)} \cdot \frac{t}{\text{arctanh } t }\leq 1+2(1-t^2)+\frac{3}{1+\frac{2}{\sqrt{1-t^2}}}.
\end{eqnarray*}
Therefore, in order to prove \eqref{ineq2}, it's sufficient for us to show that for $t \in (0,1),$
\begin{eqnarray*}
G_1(t):=\text{arctanh } t-\frac{3t}{ 1+2(1-t^2)+\frac{3}{1+\frac{2}{\sqrt{1-t^2}}} } \geq 0.
\end{eqnarray*}
We then compute the derivative of $G_1(t)$ as follows:
\begin{eqnarray*}
\frac{d}{dt} \big\{G_1(t) \big\} =\frac{\sqrt{1-t^2}B(t)+ A(t) }{2t^2(1-t^2)(t^4-3t^2+3)^2}
\end{eqnarray*}
where $A(t)=5t^{10}-6t^8-51t^6+117t^4-90t^2+27$, $B(t)=36t^6-99t^4+81t^2-27$ and 
$$A^2(t)-(1-t^2)B^2(t)=t^2(t^4-3t^2+3)^2 \cdot (25t^{10}+90t^8-309t^6+324t^4-153t^2+27).$$
Therefore $\frac{d}{dt} \big\{G_1(t) \big\}=0 $ only if $25t^{10}+90t^8-309t^6+324t^4-153t^2+27=0.$
We now claim that $p(t)=0$ has only 2 roots in $[0,1]$, where
$$p(t):=25t^{5}+90t^4-309t^3+324t^2-153t+27.$$
By a standard application of Sturm's theorem, we can find that  $p(t)$ has exactly 2 roots in $[0,1].$

Since $p(t)=0$ has only 2 roots in $[0,1],$ $G_1(t)$ then has only 2 critical points in $[0,1].$ Since $p(t^2)>0,$ for $t=0.65$ and $p(t^2)<0$, for $t=0.66$, then $G_1(t)$ has a local maxima point in $[0.65,0.66].$ Since $p(t^2)<0,$ for $t=0.94$ and $p(t^2)>0$, for $t=0.95$, then $G_1(t)$ has the other critical point in $[0.94,0.95],$ which is a local minima. We denote this unique local minima of $G_1(t)$ in $[0,1]$ by $t_m.$ 

Then it holds that
$$G_1(t_m) \geq  \text{arctanh } 0.94  -\frac{3 \cdot 0.95}{ 1+2(1-0.95^2)+\frac{3}{1+\frac{2}{\sqrt{1-0.95^2}}} } >0.$$
Here we use the fact that both $\text{arctanh }t$  and $\frac{3t}{ 1+2(1-t^2)+\frac{3}{1+\frac{2}{\sqrt{1-t^2}}} } $ are increasing in $(-1,1).$

Since $G_1(0)=0$ and $\lim_{t \rightarrow 1}G_1(t)=+\infty,$ we then conclude that $G_1(t) > 0$ for $t \in (0,1).$

\end{proof}

\subsection{Proof of One-Step Replica Symmetry}\label{section3.2}

In this section, we consider the case that the Parisi measure is 1RSB, i.e. $\mu_P=m \delta_0+(1-m) \delta_q,$ for some $0<m,q<1.$

Similar to Proposition \ref{rscriterion},  we provide a criterion for the Parisi measure to be 1-RSB based on the explicit solution of the Parisi PDE.

Consider any measure $\mu \in M[0,1]$ consisting of two atoms. Then it can be expressed by $m,q\in (0,1)$ and corresponds to the following two sequences:
\begin{eqnarray*}
&\mathbf{m}:& 0=m_0 \leq m \leq m_2 =1, \\
&\mathbf{q}:& 0=q_0 = q_1<q < q_3 =1,
\end{eqnarray*} 
such that $\mu([0,q_p])=m_p$ for $p=0,1,2.$

By Cole-Hopf transformation, the solution to the Parisi PDE is then as follows: for $q \leq u \leq 1,$
\begin{eqnarray*}
\Phi_\mu(x,u)= \log \cosh x +\frac12 \big[ \xi'(1)-\xi'(u) \big].
\end{eqnarray*}
and for $0 \leq u \leq q,$
\begin{eqnarray*}
\Phi_\mu(x,u)= \frac1m \log \mathbb{E} \big[ \cosh^m \big(x+g \sqrt{\xi'(q)-\xi'(u)} \big) \big] +\frac12 \big[ \xi'(1)-\xi'(q) \big],
\end{eqnarray*}
where $g$ is a standard Gaussian random variable.

Now for any $a>b \geq 0,$ we define $Y_{a-b}=\sqrt{\xi'(a)-\xi'(b)}.$ In particular, when $b=0$, we abbreviate $Y_{a-b}$ as $Y_a.$
By the definition of $W_\mu(u)$, we obtain that for $u \in [q,1],$
\begin{eqnarray*}
\Gamma_\mu(u)=\frac{\mathbb{E}\big[ \tanh^2 \big( Y_q  g_1+Y_{u-q} g_2 \big) \cosh \big( Y_q g_1 +Y_{u-q} g_2 \big) \cosh\big( Y_q g_1\big)^{m-1} \big] \cdot  \mathbb{E}[ \cosh \big( Y_q g\big) ]}{ \mathbb{E}[ \cosh \big( Y_q g\big)^m ] \cdot\mathbb{E}[ \cosh \big( Y_u g\big) ] },
\end{eqnarray*}
and for $u \in [0,q),$
\begin{eqnarray*}
\Gamma_\mu(u)=\frac{\mathbb{E}_1\Big[ \frac{\mathbb{E}_2[ \tanh ( Y_u  g_1+Y_{q-u} g_2 ) \cosh ( Y_u g_1+Y_{q-u} g_2 )^m ]^2}{\mathbb{E}_2[ \cosh ( Y_u  g_1+Y_{q-u} g_2 )^m ]} \Big] }{ \mathbb{E}[ \cosh \big( Y_q g\big)^m ]  }.
\end{eqnarray*}
Here $g,g_1,g_2$ are i.i.d. standard Gaussian random variables and $\mathbb{E}_i$ represents the expectation with respect to $z_i,$ for $i=1,2$.

Recall the two functions \eqref{criterionfunc1} and \eqref{criterionfunc2}, we notice that for $0 \leq u \leq q,$
\begin{eqnarray*}
&&f_\mu(u)=-\frac1m \cdot \frac{\mathbb{E}[\cosh(Y_q g)^m \log \cosh(Y_q g)]}{\mathbb{E}[\cosh(Y_q g)^m]} +\frac12[\theta(q)-\theta(u)]\\
&& +\frac{1}{m^2} \cdot  \frac{\mathbb{E}_1\big[ \mathbb{E}_2[\cosh(Y_u g_1+Y_{q-u}g_2)^m] \log \mathbb{E}_2[\cosh(Y_u g_1+Y_{q-u}g_2)^m] \big]}{\mathbb{E}[\cosh(Y_q g)^m]},
\end{eqnarray*}
and for $q\leq u \leq 1$,
\begin{eqnarray*}
&&f_\mu(u)=- \frac{\mathbb{E}[\cosh(Y_q g)^m \log \cosh(Y_q g)]}{\mathbb{E}[\cosh(Y_q g)^m]} -\frac12[\theta(u)-\theta(q)]-\frac12[\xi'(u)-\xi'(q)]\\
&& + \frac{\mathbb{E}_1\Big[ \cosh(Y_qg_1)^m \cdot \frac{\mathbb{E}_2[\cosh(Y_q g_1+Y_{u-q}g_2) \log \cosh(Y_q g_1+Y_{u-q}g_2)]}{\mathbb{E}_2[\cosh(Y_q g_1+Y_{u-q}g_2)] }\Big]}{\mathbb{E}[\cosh(Y_q g)^m]}.
\end{eqnarray*}

By Theorem \ref{thmcriterion}, the Parisi measure $\mu_P$ is 1RSB if and only if  
\begin{enumerate}
\item $ f_{\mu_P}(u) \leq 0$, for $0 \leq u \leq 1,$
\item $f_{\mu_P}(0)=f_{\mu_P}(q)=f'_{\mu_P}(q)=0$.
\end{enumerate}
Here the second condition is equivalent to the following system of equations:
\begin{equation} \label{eq1rsb}  \left\{
\begin{array}{lcl}
C^1_\beta(m,q):=-\frac{1}{m^2} \log \mathbb{E} [\cosh^m(Y_q g)] +\frac1m\frac{\mathbb{E} [\cosh^m(Y_q g) \log \cosh(Y_q g)]}{\mathbb{E} [\cosh^m(Y_q g) ]} -\frac12 \theta(q)=0,  \\
D^1_\beta(m,q):= \frac{\mathbb{E} [\tanh^2(Y_q g) \cosh(Y_q g)^m ]}{\mathbb{E} [\cosh^m(Y_q g) ]} -q=0.
\end{array} \right. \end{equation} 
Indeed we notice that $C_\beta^1(m,q)=f_{\mu_P}(0)$ and $D^1_\beta(m,q)=\Gamma_{\mu_P}(q)-q.$

Based on the computation above, we derive the following proposition to characterize the Parisi measure within the 1RSB phase:
\begin{proposition}[1RSB criterion]\label{1rsbcriterion}
For any $p \geq 3$ and $\beta>0,$ the Parisi measure $\mu_P$ is one-step replica symmetry breaking if and only if the following conditions are satisfied:
\begin{enumerate}
\item  there exists $(m_\beta,q_\beta)\in (0,1) \times (0,1)$ such that \eqref{eq1rsb} holds.
\item For $m=m_\beta$ and $q=q_\beta$, it holds that $ f_{\mu_P}(u) \leq 0$, for $0 \leq u \leq 1.$

\end{enumerate}

\end{proposition}

Now we prove the second part of Theorem \ref{mainthm} as follows:
\begin{proof}[Proof of Theorem \ref{mainthm}(2)]
Assume $p \geq 3.$
Firstly, we prove that $\beta=\beta^p_1$ is the boundary of RS and 1RSB phases by showing that when $\beta=\beta^p_1,$ the two conditions in Proposition \ref{1rsbcriterion} are satisfied except $m_1=1.$

Note that  $C^1_\beta(1,q)=C_\beta(q)$ and $D^1_\beta(1,q)=D_\beta(q).$ Recall that for any $p \geq 3,$ $(\beta^p_1,q^p_1)$ is a solution to \eqref{boundary}. Therefore when $\beta=\beta^p_1,$ $(1,q^p_1)$ is a solution to \eqref{eq1rsb} satisfying that $C^1_{\beta^p_1}(1,q^p_1)=C_{\beta^p_1}(q^p_1)=0$ and  $D^1_{\beta^p_1}(1,q^p_1)=D_{\beta^p_1}(q^p_1)=0$.

Furthermore, for any $\beta>0,$ when $m=1$, it holds that for $u,q \in [0,1],$
\begin{eqnarray*}
f_{\mu}(u) &=&-  \frac{\mathbb{E}[\cosh(Y_qg) \log \cosh(Y_qg)]}{\mathbb{E}[\cosh(Y_qg)]}+\frac{\mathbb{E}[\cosh(Y_qg) \log \cosh(Y_qg)]}{\mathbb{E}[\cosh(Y_qg )]}  \\
&& +\frac12[\xi'(q)-\xi'(u)] +\frac12[\theta(q)-\theta(u)]\\
&=& C_{\beta}(u)-C_\beta(q).
\end{eqnarray*}
In particular, when $\beta=\beta^p_1$ and $(m,q)=(1,q^p_1),$ we obtain that
$$f_{\mu_P}(u)=C_{\beta^p_1}(u) \text{ for } u \in [0,1],$$
which implies that $f_{\mu_P}(0)=f_{\mu_P}(q^p_1)=0$ and $f_{\mu_P}(u) < 0, $ for $u \in (0,q^p_1) \cup (q^p_1,1).$ Also, since $q^p_1$ is a local maximum point of $C_{\beta^p_1},$ it holds that $f''_{\mu_P}(q^p_1)<0$.

Secondly, we claim that $\frac{\partial}{\partial m} \big\{ C^1_{\beta}(m,q)\big\}_{|m=1} >0$, for any $\beta>0$ and $q \in (0,1).$
Now for any $\beta,q>0,$ we compute $\frac{\partial}{\partial m} \big\{ C^1_\beta(m,q) \big\}_{|m=1}$ as follows:
\begin{eqnarray*}
&&\frac{\partial}{\partial m} \big\{ C^1_{\beta}(m,q)\big\}_{|m=1}  \\
&=&\Big\{\frac{2}{m^3} \log \mathbb{E} [\cosh^m(Y_qg)] -\frac{2}{m^2} \frac{\mathbb{E} [\cosh^m(Y_qg) \log \cosh(Y_qg)] }{\mathbb{E} [\cosh^m(Y_qg)] }  \\
&&+ \frac{\mathbb{E} [\cosh^m(Y_qg) (\log \cosh(Y_qg))^2]}{\mathbb{E} [\cosh^m(Y_qg)]} -\frac{\mathbb{E} [\cosh^m(Y_qg) \log \cosh(Y_qg)]^2}{ \mathbb{E} [\cosh^m(Y_qg)] ^2} \Big\}_{|m=1}   \\
&=&   \xi'(q)-2 \cdot \frac{\mathbb{E} [\cosh(Y_{q}g) \log \cosh(Y_{q}g)]}{ \mathbb{E} [\cosh(Y_{q}g)] } \\
&&+ \frac{\mathbb{E} [\cosh(Y_{q}g) (\log \cosh(Y_{q}g))^2]}{\mathbb{E} [\cosh(Y_{q}g)]} -\frac{\mathbb{E} [\cosh(Y_{q}g) \log \cosh(Y_{q}g)]^2}{ \mathbb{E} [\cosh(Y_{q}g)] ^2}.
\end{eqnarray*}
Define the following function
\begin{eqnarray*}G_2(x)&:=&x-2 \cdot \frac{\mathbb{E} [\cosh(\sqrt{x} g) \log \cosh(\sqrt{x} g)]}{ \mathbb{E} [\cosh(\sqrt{x} g)] } 
\\&&+ \frac{\mathbb{E} [\cosh(\sqrt{x} g) (\log \cosh(\sqrt{x} g))^2]}{\mathbb{E} [\cosh(\sqrt{x} g)]} -\frac{\mathbb{E} [\cosh(\sqrt{x} g) \log \cosh(\sqrt{x} g)]^2}{ \mathbb{E} [\cosh(\sqrt{x} g)] ^2}\end{eqnarray*}
and  then  $\frac{\partial}{\partial m} \big\{ C^1_{\beta}(m,q)\big\}_{|m=1} =G_2\big(\xi'(q) \big).$
 Notice that $G_2(0)=0.$ In order to prove $\frac{\partial}{\partial m} \big\{ C^1_{\beta}(m,q)\big\}_{|m=1} >0,$ it then suffices for us to show that 
$G_2(x)$ is strictly increasing for $x>0.$

We then compute the derivative of $G_2(x)$ as follows:
\begin{eqnarray*}
\frac{d}{dx}\big\{G_2(x) \big\}&=&\frac{1}{\mathbb{E} [\cosh(Y_x g) ]^2} \cdot \Big\{ \mathbb{E}[\cosh(Y_xg) \log \cosh(Y_x g)] \mathbb{E}[\cosh(Y_xg)^{-1}]\\
&&-\mathbb{E}[\cosh(Y_xg)] \mathbb{E}[\cosh(Y_xg)^{-1} \log \cosh(Y_xg)] \Big\}.
\end{eqnarray*}
By considering Gaussian measure weighted by $\cosh$ and $\cosh^{-1}$, we obtain that by rearrangement inequality,
\begin{eqnarray*}
\frac{\mathbb{E}[\cosh(Y_xg)^{-1} \log \cosh(Y_xg)]} {\mathbb{E}[\cosh(Y_xg)^{-1} ]} < \mathbb{E}[\log \cosh(Y_xg)] < \frac{\mathbb{E}[\cosh(Y_xg) \log \cosh(Y_xg)]} {\mathbb{E}[\cosh(Y_xg) ]},
\end{eqnarray*}
which yields that
\begin{eqnarray*}
\mathbb{E} [\cosh(Y_x g) ]^2 \cdot  \frac{d}{dx}\big\{G_2(x) \big\}&>&  \mathbb{E}[\cosh(Y_xg) ]  \mathbb{E}[\log \cosh(Y_xg)] \mathbb{E}[\cosh(Y_xg)^{-1} ] \\
&&-\mathbb{E}[\cosh(Y_xg)] \mathbb{E}[\cosh(Y_xg)^{-1} \log \cosh(Y_xg)] \\
&>&  \mathbb{E}[\cosh(Y_xg) ]  \mathbb{E}[\log \cosh(Y_xg)] \mathbb{E}[\cosh(Y_xg)^{-1} ] \\
&&-\mathbb{E}[\cosh(Y_xg)]  \mathbb{E}[\cosh(Y_xg)^{-1} ]  \mathbb{E}[\log \cosh(Y_xg)] =0.
\end{eqnarray*}

Thirdly, we show that when $\beta$ is larger than and close to $\beta^p_1,$ \eqref{eq1rsb} must have a solution in $(0,1) \times (0,1).$

Since $D^1_{\beta^p_1}(1,q^p_1)=0$ and $D^1_{\beta^p_1}(1,q)$ is strictly decreasing with respect to $q$ near $q=q^p_1$, then there exists $\epsilon_0>0$ such that for $\beta \in [\beta^p_1,\beta^p_1+\epsilon]$,$m \in [1-\epsilon_0,1]$, $D^1_\beta(m,q)$ has exactly one zero in $[q^p_1-\epsilon_0,q^p_1+\epsilon_0]$ and it is continuous.  We then denote this unique zero by $d_{\beta,m}.$

By the computation in \eqref{cpartial}, it yields that, when $\beta=\beta^p_1,$
 \begin{eqnarray*}
\frac{\partial}{\partial \beta} \big\{ C^1_{\beta}(1,q^p_1)\big\}_{|\beta=\beta^p_1} = \frac{1}{\beta^p_1} \cdot \xi(q^p_1)>0.
 \end{eqnarray*} 
 Then by Theorem \ref{propcd1}(1), there exists $\epsilon_1 \in (0,\epsilon_0)$ such that $C^1_\beta(1,d_{\beta,1})>0$,  for $\beta \in (\beta_1^p,\beta^p_1+ \epsilon_1)$. 

Since $ C^1_{\beta^p_1}(1,q^p_1) =0$ and $\frac{\partial}{\partial m} \big\{ C^1_{\beta}(m,q)\big\}_{\big|m=1} >0$ for any $\beta,q>0$, there exists $\epsilon_2 \in (0,\epsilon_1)$ such that for any $\beta \in [\beta_1^p,\beta^p_1+ \epsilon_2]$, $m \in [1-\epsilon_2,1]$ and $q \in [q^p_1-\epsilon_2, q^p_1+\epsilon_2]$, 
it holds that $C^1_{\beta}(m,q)<0$. In particular, we obtain that $C^1_\beta(m,d_{\beta,m})<0.$

Based on the argument above, by intermediate value theorem, for any $\beta \in [\beta_1^p,\beta^p_1+ \epsilon_2]$, there exists $m_\beta \in  [1-\epsilon_2,1] $ such that $C^1_\beta(m_\beta,d_{\beta,m_{\beta}})=0$. By setting $q_\beta=d_{\beta,m_{\beta}},$ $(m_\beta,q_\beta)$ is a solution to \eqref{eq1rsb} in $(0,1) \times (0,1)$.

Lastly, we show that for any $\beta$ close enough to $\beta^p_1$, $\mu_\beta$ corresponding to $(m_\beta,q_\beta)$ satisfies $ f_{\mu_\beta}(u) \leq 0$, for $0 \leq u \leq 1,$ which is Condition (2) in Proposition \ref{1rsbcriterion}.

Note that  for any $\beta \in [\beta_1^p,\beta^p_1+ \epsilon_2)$ and $\mu_\beta$ with $(m_\beta,q_\beta)$, it holds that $$f_{\mu_\beta}(q_\beta)=f'_{\mu_\beta}(q_\beta)=0 \text{ and } f_{\mu_\beta}(0)=0.$$
Moreover, for any $\mu \in M_d[0,1]$, it holds that
\begin{eqnarray*}
\frac{d}{du}\big\{ f_{\mu}(u) \big\}&=&\frac{\xi''(u)}{2} \big( \Gamma_\mu(u)-u\big)\\
&=& \frac{\xi''(u)}{2} \Big(\big( { \mathbb{E}[ \cosh \big( Y_q g\big)^m ]  } \big)^{-1} G_3(u) -u\Big)
\end{eqnarray*}
where $G_3(u)={\mathbb{E}_1\Big[ \frac{\mathbb{E}_2[ \tanh ( Y_u  g_1+Y_{q-u} g_2 ) \cosh ( Y_u g_1+Y_{q-u} g_2 )^m ]^2}{\mathbb{E}_2[ \cosh ( Y_u  g_1+Y_{q-u} g_2 )^m ]} \Big] }$.

Since for any $f,g$ that are twice differentiable, it holds that for $p \geq 3,$
\begin{eqnarray*}
 \frac{d}{du} \Big\{ {\mathbb{E}_1\Big[ \frac{\mathbb{E}_2[ f ( Y_u  g_1+Y_{q-u} g_2 ) ]^2}{\mathbb{E}_2[ g ( Y_u  g_1+Y_{q-u} g_2 ) ]} \Big] } \Big\}_{|u=0}=0,
 \end{eqnarray*}
which implies that $f^{k}_{\mu}(0)=$ for $k=1,2,\cdots,p-2$ and $f^{p-1}_{\mu}(0)<0$.

Also recall that when  $\beta=\beta_1^p$, we obtain that $f''_{\mu_P}(q^p_1)<0$ and $f_{\mu_P}(u)<0$ for $u \in (0,q^p_1) \cup (q^p_1,1].$ Hence there exists $\epsilon \in (0,\epsilon_2)$ such that  for $\beta \in [\beta_1^p,\beta^p_1+ \epsilon)$, it holds that
 $f_{\mu_\beta}(0)=f_{\mu_\beta}(q_\beta)=0$ and  $f_{\mu_\beta}(u)<0$ for $u \in (0,q_\beta) \cup (q_\beta,1].$ By Proposition \ref{1rsbcriterion}, $\mu_\beta$ is the Parisi measure  for $\beta \in [\beta_1^p,\beta^p_1+ \epsilon)$, which is 1RSB.



\end{proof}

\bibliographystyle{abbrv}
\bibliography{biblio3}

\end{document}